\documentclass[12pt]{amsart}

\calclayout

\usepackage[dvipsnames]{xcolor}
\usepackage{caption}            
\usepackage{subcaption}         
\usepackage{multirow}           
\usepackage{listings}           
\usepackage{url}      
\usepackage[lite]{amsrefs}
\usepackage{amssymb}
\usepackage{float}    
\usepackage[all,cmtip]{xy}

\usepackage{mathrsfs}
\usepackage{tikz-cd}
\usepackage{tabularx}
\usepackage{booktabs}
\usepackage[labelfont=bf,format=plain,justification=raggedright,singlelinecheck=false]{caption}
\usepackage{bookmark}
\usepackage{hyperref}
\usepackage[color=lightgray]{todonotes}
\newcommand{\jt}{\tilde{\jmath}}
\newcommand{\ti}{\tilde{\imath}}
\newcommand{\Fm}{\mathcal{F}_m}
\newcommand{\Fnm}{\mathcal{F}_{n,m}}
\newcommand{\lcm}{\mathrm{lcm}}
\newcommand{\fl}[1]{\left\lfloor #1 \right\rfloor}
\newcommand{\cl}[1]{\left\lceil #1 \right\rceil}

\numberwithin{equation}{section}

\definecolor{airforceblue}{rgb}{0.36, 0.54, 0.66}
\definecolor{amber1}{rgb}{1.0, 0.49, 0.0}
\definecolor{amber}{rgb}{1.0, 0.75, 0.0}
\definecolor{antiquefuchsia}{rgb}{0.57, 0.36, 0.51}

\theoremstyle{plain}
\newtheorem{theorem}[equation]{Theorem}
\newtheorem{corollary}[equation]{Corollary}
\newtheorem{lemma}[equation]{Lemma}
\newtheorem{proposition}[equation]{Proposition}

\theoremstyle{definition}
\newtheorem{defn}[equation]{Definition}

\theoremstyle{remark}
\newtheorem{remark}[equation]{Remark}

\begin{document}

\title[On the $a$-number of Fermat type function fields and some of their subfields]{On the $a$-number of Fermat type function fields and some of their subfields}

\author{Marie Frank vom Braucke} \address{Department of Applied Mathematics and Computer Science, Technical University of Denmark, Kongens Lyngby 2800, Denmark}  \email{mfvbr@dtu.dk} 

\begin{abstract}
    The $a$-number of a function field $F$ is the dimension of the kernel of the Cartier operator acting on the holomorphic differentials of $F$. 
    We compute the action of the Cartier operator on the holomorphic differentials on the Hermitian function field $\mathcal{H}$, function fields of Fermat type $\Fnm$ and the Hurwitz function field $\mathcal{H}_n$. 
    From this, we obtain various results about the $a$-number of these function fields, including a generalisation of the results on the Fermat function field from Kodama and Washio~\cite{KW}.
\end{abstract}

\maketitle

\vspace{0.5cm}\noindent {\em Keywords}: Maximal function field, a-number, Cartier operator

\vspace{0.2cm}\noindent{\em MSC}: Primary: 11G20. Secondary: 11R58, 14H05.


\vspace{0.2cm}\noindent
 
\section{Introduction}
\label{sec:introduction}
Let $\mathbb{F}_{q^2}$ be a finite field of cardinality $q^2$ with $q$ a prime power.
A function field $F/\mathbb{F}_{q^2}$ with genus $g$ is called $\mathbb{F}_{q^2}$-maximal if its number of places of degree one attains the Hasse-Weil upper bound, 
\begin{align*}
    N(F)\leq q^2 + 1 + 2gq,
\end{align*}
where $N(F)$ is the number of places of degree one (also called the rational places of $F/\mathbb{F}_{q^2}$). 
Maximal function fields are widely studied in the literature, both due to theoretical interest and because of their application in coding theory and cryptography, originally due to Goppa~\cite{Goppa}.
A function field is classified by its birational invariants, and some of these influence how likely an algebraic function field is to give rise to good, long error-correcting codes. Some desirable properties of algebraic function fields in this perspective are large genera and many rational places of degree one. These values are some of the most studied birational invariants of a function field along with the $p$-rank and the automorphism group. 
The $p$-rank is the dimension of the span of the holomorphic differentials of the function field that are fixed by the Cartier operator $\mathcal{C}$, and it is connected to the size of the automorphism group (\cites{Nak, GK}), where a large automorphism group often implies that the corresponding function field has $p$-rank 0. 

A less studied birational invariant is the $a$-number, which can be defined as the dimension of the kernel of $\mathcal{C}$. 
It has been determined for infinite families of Fermat and Hurwitz function fields in~\cite{MoS} and for Fermat function fields of prime degree in~\cite{Gon}. We will present various results on the $a$-number of the Fermat type and Hurwitz function fields, and since the regular Fermat function fields also fall within this class of function fields, we recover and generalise the results obtained by Kodama and Washio in~\cite{KW}. In fact we prove that the relations found in~\cite{KW} for the classical Fermat function fields are true for all the classes of function fields studied in this paper. 

The paper is divided into two main parts after the preliminary section.
In Section~\ref{sec:C_on_ff}, we give explicit expressions for the Cartier operator on holomorphic differentials on the Hermitian function field, which we extend to results on maximal function fields of Fermat type.
In Section~\ref{sec:misc_results}, we will show various relations between $a$-numbers of Fermat type and Hurwitz function fields of different characteristic. 

\section{Preliminaries}
\label{sec:preliminaries}
In this section, we will present the preliminaries needed for the rest of the paper. The first subsection is dedicated to presenting the function fields that we will treat, and in the second we will recall the definition of, and some results on, the Cartier operator and holomorphic differentials. Most of the results on the Cartier operator was presented in 1957 by Cartier in~\cite{Cartier}, however the reader is referred to~\cite{MoS} for a more recent source. 
Before that however, we will bring a short theorem that will be of much use later.
\begin{theorem}[Lucas's Theorem~\cite{Lucas}]\label{thm:Lucas}
    Let $p$ be a prime, and $a,b$ nonnegative integers with the $p$-ary expansions $a=a_0+a_1p + ... + a_kp^k$ and $b=b_0+b_1p + ... + b_kp^k$. Then
    \begin{align*}
        \binom{a}{b} \equiv \prod_{i=0}^{k}\binom{a_i}{b_i} \mod{p}.
    \end{align*}
\end{theorem}

\subsection{Fermat type and Hurwitz function fields}
Let $p$ be a prime and $q=p^e$ for a positive integer $e$. 
Then the Hermitian function field $\mathcal{H}=\mathbb{F}_{q^2}(X,Y)$ is given by 
\begin{align*}
    X^{q+1}+Y^{q+1}+1=0.
\end{align*}
$\mathcal{H}$ has genus $\frac{q(q-1)}{2}$ and is $\mathbb{F}_{q^2}$-maximal.
The Fermat type function field $\mathcal{F}_{n,m}$ is given by
\begin{align*}
    U^n+V^m+1=0,
\end{align*}
with $p\nmid n$ and $p\nmid m$. From~\cite[Example 6.3.3]{Sti}, it follows that $\Fnm$ has genus $g(\Fnm)=\frac{(n-1)(m-1)+1-\gcd(m,n)}{2}$, and in~\cite[Theorem 5]{TT1} it is shown that $\Fnm$ is $\mathbb{F}_{q^2}$-maximal if and only if both $n,m$ divide $q+1$. 
Assume this is the case and denote by $\tilde{m}=\frac{q+1}{m}$ and $\tilde{n}=\frac{q+1}{n}$. Set $U=X^{\tilde{n}}$ and $V=Y^{\tilde{m}}$, then $\Fnm=\mathbb{F}_{q^2}(U,V)$ is a subfield of the Hermitian function field as $U^n + V^m + 1 = X^{q+1} + Y^{q+1} + 1 = 0$. 
$\Fnm$ is a generalization of the Fermat function field $\Fm$ which is obtained by setting $n=m$ and is amongst the most studied function fields in the literature.
The last function field we will treat is the Hurwitz function field $\mathcal{H}_n=\mathbb{F}_{q^2}(S,T)$ with defining equation
\begin{align*}
    S^nT+T^n+S=0,
\end{align*}
and, since it is non-singular, its genus is $g(\mathcal{H}_n)=\frac{n(n-1)}{2}$. 
$\mathcal{H}_n$ is $\mathbb{F}_{q^2}$-maximal exactly when $q+1\equiv 0 \mod (n^2-n+1)$ by~\cite[Theorem 3.1]{AKT}, and if we denote by $N=n^2-n+1$, then the Fermat function field $\mathcal{F}_N=\mathbb{F}_{q^2}(U,V)$ defined by $U^N+V^N+1=0$ is an extension of $\mathcal{H}_n$. Here we have $S=U^n/V$, $T=UV^{n-1}$ (\cite[Lemma 3.2]{AKT}). 

\subsection{The Cartier operator and holomorphic differentials}
Let $F/K$ be a function field over an algebraically closed field of characteristic $p>0$ of genus $g$, and let $t$ be separating element of $F/K$.
We denote by $\Delta(F/K)$ the $F$-module of differentials of $F/K$.
Every differential $\omega\in\Delta(F/K)$ can be written as $\omega=xdt$ with $x\in F$, in particular $\omega$ has unique representation
\begin{align*}
    \omega=(x_0^p + x_1^pt + \cdots + x_{p-1}^p t^{p-1})dt,
\end{align*}
with $x_i\in F$ for $i=0,\ldots p-1$.
The Cartier operator $\mathcal{C}: \Delta(F/K)\to \Delta(F/K)$ is defined by $\mathcal{C}(\omega)=x_{p-1}dt$. Note that $\mathcal{C}$ is independent of the choice of $t$, a result by Tate in~\cite{Tate}.
The Cartier operator satisfies the following properties:
\begin{itemize}
    \item $\mathcal{C}$ is $p^{-1}$-linear, i.e. $\mathcal{C}(\omega+\omega')=\mathcal{C}(\omega)+\mathcal{C}(\omega')$ and $\mathcal{C}(x^p\omega)=x\mathcal{C}(\omega)$,
    \item $\mathcal{C}$ is surjective,
    \item $\mathcal{C}(\omega)=0$ if and only if $\omega$ is exact,
    \item $\mathcal{C}(dx/x)=dx/x$ for all $x\neq 0$,
\end{itemize}
for $\omega,\omega'\in\Delta(F/K)$ and $x\in F$. A differential $\omega$ is exact if it is of the form $dx$ for some $x\in F/K$. 
We call a differential $\omega$ holomorphic (or of the first kind) if $\text{div}(\omega)$ is effective and denote by $D_1$ the set of holomorphic differentials. Then we have that $\mathcal{C}(D_1)\subseteq D_1$, and $\mathcal{C}^g(D_1)= \mathcal{C}^{g+k}(D_1)$ for any $k\in\mathbb{N}$. 
We define the $p$-rank (or Hasse-Witt invariant) of $F/K$ as $\gamma(F) = \dim_K(\mathcal{C}^g(D_1))$, which is an invariant of the function field.
We have $0\leq \gamma\leq g$, where $F/K$ is called ordinary if $\gamma=g$. An $\mathbb{F}_{q^2}$-maximal function field has $\gamma=0$. 
Another invariant in relation to $\mathcal{C}$ is the $a$-number, $a=\text{dim}_K \ker \mathcal{C}$. The $a$-number is related to the $p$-rank by $0\leq a+\gamma \leq g$. 
The following theorem is from~\cite[Theorem 3.3]{GT}.
\begin{theorem}\label{thm:C^n=0}
    For an $\mathbb{F}_{q^2}$-maximal (or minimal) function field with $q=p^n$, we have $\mathcal{C}^n=0$. 
\end{theorem}
Let $f$ be the defining equation for $F/K=K(x,y)/K$ with $d=\deg f$, and assume that $f$ is non-singular. 
Due to a theorem from Gorenstein~\cite[Theorem 12]{Go}, we then have that $D_1 = \left\{ \frac{h(x,y)}{f_y(x,y)}dx \mid \deg h \leq d-3\right\}$, in particular we have a basis for the space of holomorphic differentials given by
\begin{align*}
    \left\{ \frac{x^iy^j}{f_y(x,y)}dx \mid i+j \leq d-3\right\}.
\end{align*}
Note that in fact this agrees with
\begin{align*}
    \dim_K(D_1)=g = \frac{(d-1)(d-2)}{2}=\binom{d-1}{2}=\#(\text{monomials of degree $\leq d-3$ in 2 variables).}
\end{align*}
Define the differential operator $\nabla=\frac{\partial^{2p-2}}{\partial x^{p-1}\partial y^{p-1}}$. Then the following theorem states an explicit formula to compute the action of the Cartier operator on a holomorphic differential which was found by Stöhr and Voloch in~\cite{SV}. 
\begin{theorem}\label{thm:S-V}
    With the above assumptions, we have that for every $h\in F/K$,
    \begin{align*}
        \mathcal{C}\left(\frac{hdx}{f_y}\right) = \left(\nabla(f^{p-1}h)\right)^{1/p}\frac{dx}{f_y}.
    \end{align*}
\end{theorem}
They also note that 
\begin{align}\label{eq:nabla}
    \nabla\left(\sum_{i,j} c_{ij} x^i y ^j \right) = \sum_{i,j}c_{ip+p-1,jp+p-1} x^{ip} y^{jp},
\end{align}
which will be useful in obtaining an expression for the rank of the Cartier operator. 

\section{The Cartier operator on various function fields}
\label{sec:C_on_ff}
The aim of this section is to show that for the subfields of the Hermitian function field that we introduced earlier, the Cartier operator $\mathcal{C}$ maps a holomorphic differential $\omega$ to another uniquely determined holomorphic differential $\tilde{\omega}$. We will compute an explicit expression for $\tilde{\omega}$, which we will use in the following sections. 

\subsection{The Cartier operator on the Hermitian function field}
A basis for the holomorphic differentials of $\mathcal{H}$ is given by 
\begin{align*}
    \left\{\omega_{i,j}=\frac{X^i Y^j}{X\ Y\ (q+1)Y^{q}}dX \mid 2\leq i+j < q+1 \right\}.
\end{align*}
Before presenting the main theorem of this section, we need a lemma. 
\begin{lemma}\label{lemma:equal_ceilings}
        For an integer $1\leq k\leq p^e$, we have $\cl{\frac{kp^{e-1}}{p^e}} = \cl{\frac{kp^{e-1}}{p^e+1}}$.
    \end{lemma}
    \begin{proof}
        For $k=p^e$, we have $\cl{\frac{kp^{e-1}}{p^e}}=p^{e-1}$, and $\cl{\frac{kp^{e-1}}{p^e+1}}=\cl{\frac{p^{e-1}(p^e+1)-p^{e-1}}{p^e+1}}=p^{e-1}$. Otherwise write $k=k_0+k_1p+...+k_{e-1}p^{e-1}$ with $0\leq k_i\leq p-1$ for $i=0,...,e-1$. Then 
        \begin{align*}
            \cl{\frac{kp^{e-1}}{p^e}}= \cl{\frac{k}{p}} = \cl{\frac{k_0}{p}} + k_1 + ... + k_{e-1}p^{e-2}.
        \end{align*}
        On the other hand, writing $kp^{e-1} = k_0p^{e-1} + (p^e+1)(k_1 + ...+k_{e-1}p^{e-2}) - (k_1 + ...+k_{e-1}p^{e-2})$, we get
        \begin{align*}
            \cl{\frac{kp^{e-1}}{p^e+1}}
            &= \cl{\frac{k_0p^{e-1} - (k_1 + ...+k_{e-1}p^{e-2})}{p^e+1}} + k_1 + ...+k_{e-1}p^{e-2}. 
        \end{align*}
        Finally, we have
        \begin{align*}
         \cl{\frac{k_0p^{e-1} - (k_1 + ...+k_{e-1}p^{e-2})}{p^e+1}}
        =\begin{cases}
            0 & \text{if } k_0=0\\
            1 & \text{otherwise}
        \end{cases},   
        \end{align*}
        which is exactly equal to $\cl{\frac{k_0}{p}}$.
    \end{proof}

\begin{theorem}\label{thm:C(omega)}
    For positive integers $i,j$ such that $i+j < q+1$, we have $\mathcal{C}(\omega_{i,j})=0$ exactly when $i,j\not \equiv 0\mod p$ and $(i \mod p) + (j \mod p) \leq p$. Otherwise, we have 
    \begin{align*}
        \mathcal{C}(\omega_{i,j}) = (-1)^{\cl{\frac{jp^{e-1}}{q+1}}+j} \binom{(j-1) \mod p}{(p-i) \mod p}\ \omega_{\ti,\jt},
    \end{align*}
    where $\ti=(q+1)\cl{\frac{ip^{e-1}}{q+1}} -ip^{e-1}$ and $\jt = (q+1)\cl{\frac{jp^{e-1}}{q+1}} -jp^{e-1}$.
\end{theorem}
\begin{proof}
    Write $\omega_{i,j} = \frac{X^{i-1} Y^j(1+X^{q+1})^{p-1}}{(Y^{q+1})^p}dX$, and use $1=q+1-p^e$ to write $Y^j=(Y^{q+1})^{j}(Y^{-jp^{e-1}})^p$. 
    Then we can get $\omega_{i,j}$ on the form
    \begin{align*}
        \omega_{i,j} 
        &=\frac{X^{i-1} (-1-X^{q+1})^{j}(Y^{-jp^{e-1}})^p (1+X^{q+1})^{p-1}}{(Y^{q+1})^p}  dX\\
        &= (-1)^{j} \left(\frac{Y^{-jp^{e-1}}  }{Y^{q+1}}\right)^p X^{i-1} (1+X^{q+1})^{p-1+j} dX\\
        &= \left((-1)^{j} \frac{Y^{-jp^{e-1}}  }{Y^{q+1}}\right)^p\ \sum_{\ell=0}^{p-1+j} \binom{p-1+j}{\ell}X^{(q+1)\ell+i-1} dX.
    \end{align*}
    From this, we extract the $(p-1)$'th powers of $X$ to get an expression for $\mathcal{C}(\omega_{i,j})$. Note that $(q+1)\ell+i-1\equiv p-1 \mod p$ exactly when $\ell\equiv -i \mod p$, so we get
    \begin{align*}
        \mathcal{C}(\omega_{i,j}) 
        &= (-1)^{j} \frac{Y^{-jp^{e-1}}}{Y^{q+1}}\ \sum_{\alpha=\cl{\frac{i}{p}}}^{\fl{\frac{i+j-1}{p}}+1}\binom{p-1+j}{\alpha p -i}X^{(q+1)\alpha-ip^{e-1}-1} dX,
    \end{align*}
    where the exponent of $X$ comes from $\frac{(q+1)(\alpha p-i)+i-1-(p-1)}{p}$. 
    In order to have a positive exponent of $Y$, take $\beta_j$ to be the minimal nonnegative integer such that $(q+1)\beta_j -jp^{e-1}\geq 0$, i.e. $\beta_j = \cl{\frac{jp^{e-1}}{q+1}}$. Then write $Y^{-jp^{e-1}} = \frac{Y^{-jp^{e-1}}Y^{(q+1)\beta_j}}{Y^{(q+1)\beta_j }}= (-1)^{\beta_j}\frac{Y^{(q+1)\beta_j -jp^{e-1}}}{(1+X^{(q+1)})^{\beta_j}}$, and we obtain 
    \begin{align*}
        \mathcal{C}(\omega_{i,j})  
        = \left((-1)^{\left\lceil\frac{jp^{e-1}}{q+1}\right\rceil+j} \frac{Y^{(q+1)\left\lceil\frac{jp^{e-1}}{q+1}\right\rceil-jp^{e-1}} X^{(q+1)\cl{\frac{i}{p}} -ip^{e-1}-1}}{Y^{q+1}}\  \frac{\sigma}{(1+X^{q+1})^{\left\lceil\frac{jp^{e-1}}{q+1}\right\rceil}} \right) dX,
    \end{align*}
    where $\sigma=\sum_{\alpha=0}^{\fl{\frac{i+j-1}{p}} - \cl{\frac{i}{p}} +1}\binom{p-1+j}{\left(\alpha+\cl{\frac{i}{p}}\right) p -i}X^{(q+1)\alpha}$. 

    The first natural question is when in fact this expression is zero modulo $p$, so let us take a closer look at the binomial coefficients $\binom{p-1+j}{\left(\alpha+\cl{\frac{i}{p}}\right) p -i}$. 
    From Theorem \ref{thm:Lucas} (Lucas) we get that every one of these coefficients is equivalent to a multiple of $\binom{p-1+j}{\cl{\frac{i}{p}} p -i}$ modulo $p$, so let us start by determining when $\binom{p-1+j}{\cl{\frac{i}{p}} p -i} \equiv 0 \mod p$. Note that $i,j < p^e$, so we can write their $p$-ary expansions as $i = a_0 + a_1p + ... + a_{e-1}p^{e-1}$ and $j = b_0 + b_1 p + ... +b_{e-1}p^{e-1}$ with $0\leq a_i,b_i \leq p-1$ for $i=0,...,e-1$. Thus we have $\left\lceil\frac{i}{p}\right\rceil =  \left\lceil\frac{a_0}{p}\right\rceil + a_1 + ... + a_{e-1}p^{e-2}$, and
    \begin{align*}
        \left\lceil\frac{i}{p}\right\rceil p - i =  \left\lceil\frac{a_0}{p}\right\rceil p - a_0 = 
        \begin{cases}
            0 & \text{if } a_0=0\\
            p-a_0 & \text{otherwise}.
        \end{cases}
    \end{align*}
    In particular $\binom{p-1+j}{\cl{\frac{i}{p}} p -i}\not\equiv 0 \mod p$ when $a_0=0$. 
    We also get the $p$-ary expansion
    \begin{align*}
        p-1+j =
        \begin{cases}
            (p-1) + b_1 p + ... +b_{e-1}p^{e-1} & \text{if } b_0 = 0\\
            (b_0-1) + (b_1+1) p+ ... +b_{e-1}p^{e-1} & \text{otherwise}
        \end{cases}.
    \end{align*}
    Hence when $b_0=0$, we have $\binom{p-1+ j}{\cl{\frac{ i}{p}} p - i} \not\equiv 0 \mod p$, since $p-1\geq \left\lceil\frac{ i}{p}\right\rceil p -  i$. 
    If instead $a_0\neq 0$ and $b_0\neq 0$, Theorem \ref{thm:Lucas} gives us that $\binom{p-1+ j}{\cl{\frac{ i}{p}} p - i}\equiv \binom{b_0-1}{p-a_0} \mod p$, which is congruent to 0 modulo $p$ if and only if $p-a_0 > b_0-1$. Note that $b_0-1 < p$, so $\binom{b_0-1}{p-a_0}$ cannot be a multiple of $p$. In short, we have $\binom{p-1+ j}{\cl{\frac{ i}{p}} p - i} \equiv 0 \mod p$ whenever $a_0\neq0$, $b_0\neq0$ and $a_0+b_0\leq p$. 
    As we wish to compute an expression for $\mathcal{C}(\omega_{i,j})$, we will therefore for the rest of the argument assume that $a_0=0$, or $b_0=0$, or $a_0+b_0>p$.

    From the expressions above, we then get $\alpha p + \cl{\frac{ i}{p}} p - i = \alpha p + ((p-a_0) \mod{p})$ and $p-1+ j= \cl{\frac{ j}{p}}p + ((b_0-1) \mod{p})$, and we obtain
    \begin{align*}
        \binom{p-1+ j}{\alpha p + \cl{\frac{ i}{p}} p - i}
        \equiv \binom{(b_0-1) \mod{p}}{(p-a_0) \mod{p}}\binom{\cl{\frac{ j}{p}}p}{\alpha p} 
        \equiv \binom{(b_0-1) \mod{p}}{(p-a_0) \mod{p}}\binom{\cl{\frac{ j}{p}}}{\alpha } \mod{p},
    \end{align*}
    both equivalences by Theorem \ref{thm:Lucas}. The upper bound for $\sigma$ may be written as
    \begin{align*}
        \fl{\frac{ i+ j-1}{p}} - \cl{\frac{ i}{p}} +1
        = \cl{\frac{ i+ j}{p}} - \cl{\frac{ i}{p}} 
        =  \cl{\frac{a_0+b_0}{p}} + b_1+ ... +b_{e-1}p^{e-2}-\cl{\frac{a_0}{p}}.
    \end{align*}
    Under our previous assumption, we consider the following
    \begin{align*}
         \cl{\frac{a_0+b_0}{p}}-\cl{\frac{a_0}{p}} + b_1 + ... +b_{e-1}p^{e-2}  = 
         \begin{cases}
            \cl{\frac{b_0}{p}} + b_1 + ... +b_{e-1}p^{e-2}  &\text{if } a_0=0\\
            0 + b_1 + ... +b_{e-1}p^{e-2} &\text{if } b_0=0\\
            1+b_1 + ... +b_{e-1}p^{e-2} &\text{if } a_0+b_0 > p.
         \end{cases}
    \end{align*}
    In particular, the upper bound of $\sigma$ is simply $\cl{\frac{ j}{p}}$, so we may write 
    \begin{align*}
        \sigma = \sum_{\alpha=0}^{\cl{\frac{ j}{p}}}\binom{(b_0-1) \mod p}{(p-a_0) \mod p}\binom{\cl{\frac{ j}{p}}}{\alpha}X^{n\alpha} 
        = \binom{(b_0-1) \mod p}{(p-a_0) \mod p} \left(1+X^n\right)^{\cl{\frac{ j}{p}}}.
    \end{align*}
    With Lemma~\ref{lemma:equal_ceilings}, we get that $\cl{\frac{jp^{e-1} }{p^e+1}}=\cl{\frac{j}{p}}$, in particular $\frac{\sigma}{(1+X^n)^{\left\lceil\frac{jp^{e-1}}{m}\right\rceil}}$ is simply equal to the binomial coefficient $\binom{(b_0-1) \mod p}{(p-a_0) \mod p}$. Hence
    \begin{align*}
        \mathcal{C}(\omega_{i,j})  
        &= (-1)^{\left\lceil\frac{jp^{e-1}}{q+1}\right\rceil+ j}
        \binom{(b_0-1) \mod p}{(p-a_0) \mod p}
        \left(\frac{Y^{(q+1)\left\lceil\frac{jp^{e-1}}{q+1}\right\rceil-jp^{e-1}} X^{(q+1)\cl{\frac{ip^{e-1}}{q+1}} -ip^{e-1}-1}}{Y^{q+1}}\right)dX.
    \end{align*}
    What we need to conclude the proof is to show that if we define $\ti=(q+1)\cl{\frac{ip^{e-1}}{q+1}} -ip^{e-1}$ and $\jt =(q+1)\left\lceil\frac{jp^{e-1}}{q+1}\right\rceil-jp^{e-1}$, then these exponents are unique and satisfy the bound $\ti+\jt<q+1$.
    To show uniqueness, assume that there exist positive integers $i_1,i_2 <q+1$ such that $(q+1)\cl{\frac{i_1p^{e-1}}{q+1}} -i_1p^{e-1}=(q+1)\cl{\frac{i_2p^{e-1}}{q+1}} -i_2p^{e-1}$. Then $i_1p^{e-1}\equiv i_2p^{e-1} \mod (q+1)$, in particular $i_1\equiv i_2 \mod{(q+1)}$. Equality now follows since $i_1,i_2<q+1$. The argument is similar for $\jt$. 

    Finally, we show that $\ti+\jt < q+1$ if and only if either $a_0+b_0>p$, or $a_0=0$ or $b_0=0$. By Lemma~\ref{lemma:equal_ceilings}, we can write $\ti+\jt = (p^e+1)\left(\cl{\frac{ i}{p}}+\cl{\frac{ j}{p}}\right) -( i+ j)p^{e-1}$, and cancel terms to get
    \begin{align}\label{eq:i,j_tilde_bound}
        \ti+\jt 
        &= \left(\cl{\frac{a_0}{p}}+\cl{\frac{b_0}{p}}\right)(p^e+1) + (a_1+b_1) + ... + (a_{e-1}+b_{e-1})p^{e-2} -(a_0+b_0)p^{e-1}.
    \end{align}
    Note that $i+j = (a_0+b_0)+(a_1+b_1)p+...+(a_{e-1}+b_{e-1})p^{e-1}\leq p^e$, so in the case where $a_0+b_0>p$, we get that 
    $(a_1+b_1)+...+(a_{e-1}+b_{e-1})p^{e-2}
    <p^{e-1}-1$,
    as well as $\cl{\frac{a_0}{p}}+\cl{\frac{b_0}{p}}=2$. 
    In this case, Equation~\eqref{eq:i,j_tilde_bound} implies that
    \begin{align*}
        \ti+\jt
        \leq 2 (p^e+1) +  (p^{e-1}-2) -(p+1)p^{e-1} 
        = p^e < q + 1.
    \end{align*}
    If $a_0=0, b_0\neq 0$, then $ i+ j \leq p^e$ gives us $(a_1+b_1)+...+(a_{e-1}+b_{e-1})p^{e-2}< p^{e-1}$, so 
    \begin{align*}
        \ti + \jt 
        \leq (p^e+1) + p^{e-1}-1 -b_0p^{e-1} < q+1.
    \end{align*}
    The case $a_0\neq0,b_0=0$ is completely analogous, and the case when both $a_0=b_0=0$ follows directly from Equation~\eqref{eq:i,j_tilde_bound} and $i+j<q+1$. 
    Now, if instead $a_0+b_0\leq p$ and $a_0, b_0\neq 0$, we have $\cl{\frac{a_0}{p}}+\cl{\frac{b_0}{p}}=2$ and Equation~\eqref{eq:i,j_tilde_bound} immediately gives us that
    \begin{align*}
        \ti+\jt 
        \geq 2(p^e+1) + (a_1+b_1) + ... + (a_{e-1}+b_{e-1})p^{e-2} - p^{e}
        > q+1. 
    \end{align*}
\end{proof}

\begin{remark}\label{rem:ZN}
    Theorem~\ref{thm:C(omega)} allows us to directly compute the dimension of the kernel of the Cartier operator using that $\mathcal{C}(\omega_{i,j})= 0$ exactly when $i,j\not\equiv 0 \mod{p}$ and $(i \mod{p}) + (j \mod{p}) \leq p$.
    Using this, one can e.g. show that for a maximal Fermat type function field $\Fnm$, we have $a(\Fnm)=g(\Fnm)/2$ when $q=p^2$. We will, however, bring a shorter proof of this in Section~\ref{sec:misc_results}. 
\end{remark}

\subsection{The Cartier operator on the function field $\mathcal{F}_{n,m}$.}\label{sec:Fnm}
In this section, we will only consider the case in which $\Fnm$ is maximal, so recall that in this case, $\Fnm$ is a subfield of the Hermitian function field via $U=X^{\tilde{n}}$ and $V=Y^{\tilde{m}}$, and note that $\frac{dU}{dX} = \delta_X(X^{\tilde{n}})=\tilde{n}X^{\tilde{n}-1}$. 
A basis for the holomorphic differentials on $\mathcal{F}_{n,m}$ is 
\begin{align*}
    \left\{ \omega_{i,j} = \frac{U^i V^j}{U\ V\ mV^{m-1}}dU \mid 0< i,j\text{ and } mi+nj < mn\right\},
\end{align*}
and we see that 
\begin{align*}
    \omega_{i,j} 
    = \frac{(X^{\tilde{n}})^{i-1} (Y^{\tilde{m}})^j}{m(Y^{\tilde{m}})^{m}}\frac{dU}{dX}dX
    = \tilde{m}\tilde{n}\frac{X^{\tilde{n}i-1} Y^{\tilde{m}j}}{(q+1)Y^{q+1}}dX
    =\tilde{m}\tilde{n} \bar{\omega}_{\tilde{n}i,\tilde{m}j},
\end{align*}
where $\bar{\omega}_{\tilde{n}i,\tilde{m}j}$ is a holomorphic differential on $\mathcal{H}$.
Indeed, if $mi+nj<nm$, then $nm(\tilde{n}i + \tilde{m}j) = (q+1)mi + (q+1)nj < (q+1)nm$ implying that $\tilde{n}i + \tilde{m}j<q+1$ and $\bar{\omega}_{\tilde{n}i,\tilde{m}j}$ is holomorphic.

\begin{corollary}\label{cor:fermat_type_C(w_ij)}
    For positive integers $i,j$ such that $mi+nj < mn$, we have $\mathcal{C}(\omega_{i,j})=0$ exactly when $i,j\not \equiv 0\mod p$ and $(\tilde{n}i \mod p) + (\tilde{m}j \mod p) \leq p$. Otherwise, we have 
    \begin{align*}
        \mathcal{C}(\omega_{i,j}) = (-1)^{\cl{\frac{jp^{e-1}}{m}}+\tilde{m}j} \binom{(\tilde{m}j-1) \mod p}{(p-\tilde{n}i) \mod p}\ \omega_{\ti,\jt},
    \end{align*}
    where $\ti=n\cl{\frac{ip^{e-1}}{n}} -ip^{e-1}$ and $\jt = m\cl{\frac{jp^{e-1}}{m}} -jp^{e-1}$.
\end{corollary}
\begin{proof}
    From Theorem~\ref{thm:C(omega)}, and since $\tilde{m},\tilde{n}$ are integers, we have
    \begin{align*}
        \mathcal{C}(\tilde{m}\tilde{n}\bar{\omega}_{\tilde{n}i,\tilde{m}j})
        = (-1)^{\cl{\frac{jp^{e-1}}{m}}+\tilde{m}j} \binom{(\tilde{m}j-1) \mod p}{(p-\tilde{n}i) \mod p}\ \tilde{m}\tilde{n} \bar{\omega}_{\tilde{n}(n\cl{\frac{ip^{e-1}}{n}} -ip^{e-1}),\tilde{m}(m\cl{\frac{jp^{e-1}}{m}} -jp^{e-1})}.
    \end{align*}
    Thus, once we note that $\tilde{n}i\equiv 0 \mod p \iff i \equiv 0 \mod p$ (and equivalently for $j$), the result follows.
\end{proof}

\subsection{The Cartier operator on the Hurwitz function field.} 
Also in this section, will we assume that $\mathcal{H}_n=\mathbb{F}_{q^2}(S,T)$ is maximal, that is, $q+1\not\equiv 0 \mod N$ with $N=n^2-n+1$, and $\mathcal{H}_n$ is a subfield of the (maximal) Fermat function field $\mathcal{F}_{N}=\mathbb{F}_{q^2}(U,V)$ with $S=U^n/V$ and $T=UV^{n-1}$. Thus
\begin{align*}
    \frac{dS}{dU} = \delta_U\left(\frac{U^n}{V}\right) 
    = \frac{nVU^{n-1}-U^n\delta_U(V)}{V^2} 
    = \frac{nU^{n-1}}{V} + \frac{U^nU^{N-1}}{V^2V^{N-1}}  
    = n\frac{U^{n-1}}{V} + \frac{U^{n^2}}{V^{N+1}},
\end{align*}
and a basis for the holomorphic differentials on $\mathcal{H}_n$ is
\begin{align*}
    \left\{ \omega_{i,j}=\frac{S^i T^j}{S\ T\ (S^n+nT^{n-1})} \mid 0<i,j \text{ and } i+j < n+1\right\}.
\end{align*}
Similarly to the previous case, we wish to lift it to $\mathcal{F}_N$. Consider
\begin{align*}
    \omega_{i,j}
    &= \frac{U^{n(i-1)}/V^{i-1} U^{j-1}V^{(n-1)(j-1)}}{U^{n^2}/V^n+nU^{n-1}V^{(n-1)^2}} \frac{dS}{dU}dU\\
    &= \frac{U^{n(i-1)+j}V^{(n-1)j-i+1}}{\frac{U^n}{V}\left(U^N + nV^N\right)} \left(n\frac{U^{n-1}}{V} + \frac{U^{n^2}}{V^{N+1}}\right)dU\\
    &= \frac{U^{n(i-1)+j}V^{(n-1)j-i+1}}{U\ V\ V^{N-1}} dU\\
    &= N\cdot \bar{\omega}_{n(i-1)+j,(n-1)j-i+1},
\end{align*}
where $\bar{\omega}_{n(i-1)+j,(n-1)j-i+1}$ is a holomorphic differential on $\mathcal{F}_N$; indeed we have that both $n(i-1)+j>0$ and $(n-1)j-i+1>0$ as well as 
\begin{align*}
    n(i-1)+j+(n-1)j-i+1 =n(i+j)-n+1-i \leq  n^2-n+1-i < N.
\end{align*}

\section{Various results about the $a$-number of the function field $\mathcal{F}_{n,m}$.}
\label{sec:misc_results}

In this section, we will study the $a$-number of the Fermat type function field $\Fnm$ over base fields of different characteristics. 
From now on, we will therefore write $\Fnm^{(p)}$ for the Fermat type function field defined by $U^n+V^m+1$ for positive integers $n,m$ over a base field of characteristic $p>0$. 
Recall that $\Fnm^{(p)}$ is a generalization of the Fermat and Hermitian function fields, and our results all extend in a natural way to these.
In the article~\cite{MoS}, it is shown that the rank of the Cartier operator on the Fermat function field $\mathcal{F}_m$ is dependent only on $m$ and the characteristic of the base field. We will generalise this result to $\mathcal{F}_{n,m}^{(p)}$ in the following. 
\begin{proposition}\label{prop:cong_F_nm}
    The rank of the Cartier operator on the function field $\mathcal{F}_{n,m}^{(p)}$ 
    is equal to the number of pairs of positive integers $(i,j)$ with $mi+nj < nm$ such that there exists a solution $(h,k)$ with $0\leq h \leq p-1$ and $0\leq k \leq h$ to 
\begin{align*}
    \begin{cases}
        n(p-1-h)+i &\equiv 0 \mod p\\
        mk+j & \equiv 0 \mod p.
    \end{cases}
\end{align*}
\end{proposition}
\begin{proof}
    The proof can be taken almost verbatim from that of~\cite[Proposition 3.1]{MoS} with the only difference being the exponent of $U$ and an offset by 1 in the integers $i$ and $j$. We bring it here for completion.

    Recall that we have a basis for the holomorphic differentials on $\Fnm$ with elements of the form $\omega_{i,j}=\frac{U^{i-1} V^{j-1}}{mV^{m-1}}$ with $0<i,j$ and $mi+nj<nm$. Theorem~\ref{thm:S-V} gives us that if we apply $\mathcal{C}$ to such a basis element, we get
    \begin{align*}
        \mathcal{C}\left(\omega_{i,j}\right)= (\nabla ( (U^n+V^m+1)^{p-1} U^{i-1} V^{j-1})^{1/p} dU/(mV^{m-1}),
    \end{align*}
    so we will apply $\nabla$ to 
    \begin{align*}
        (U^n+V^m+1)^{p-1} U^{i-1} V^{j-1} = \sum_{h=0}^{p-1}\sum_{k=0}^{h} \binom{p-1}{h}\binom{h}{k} U^{n(p-1-h)+i-1} V^{mk+j-1}.
    \end{align*}
    Equation~\eqref{eq:nabla} states that $\nabla$ kills all monomials not on the form $U^{ap+p-1}V^{bp+p-1}$ for some integers $a,b$. Therefore, $\nabla \left( (U^n+V^m+1)^{p-1} U^{i-1} V^{j-1} \right)\neq 0 $ if and only if some $(h,k)$ with  $0\leq h \leq p-1$ and $0\leq k \leq h$ satisfies both congruences 
    \begin{align*}
        \begin{cases}
        n(p-1-h)+i &\equiv 0 \mod p\\
        mk+j & \equiv 0 \mod p.
    \end{cases}
    \end{align*}
    Let $(i,j)\neq(i_0,j_0)$ be pairs of positive integers such that both $mi+nj<nm$ and $mi_0+nj_0<nm$, and $\nabla \left( (U^n+V^m+1)^{p-1} U^{i-1} V^{j-1} \right)$ and $\nabla \left( (U^n+V^m+1)^{p-1} U^{i_0-1} V^{j_0-1} \right)$ are nonzero. Then if we can show that they are linearly independent over the base field, the result follows. 
    We will show this by proving that they do not share any monomial, so assume for contradiction that there exists solutions $(h,k)$ and $(h_0,k_0)$ such that 
    \begin{align*}
        \begin{cases}
        n(p-1-h)+i-1= n(p-1-h_0) + i_0-1,\\
        mk+j-1 = mk_0 + j_0-1.
        \end{cases}
    \end{align*}
    If $j=j_0$, then $i\neq i_0$, which implies $h\neq h_0$, so we may assume $h>h_0$. However, this implies that $i-i_0=n(h-h_0)>n$, a contradiction. Otherwise $j\neq j_0$, so $k\neq k_0$, and may assume $k>k_0$. This yields $j_0-j = m(k-k_0)>m$, also a contradiction. 
    We can therefore conclude that the number of basis elements $\omega_{i,j}$ such that $\nabla \left( (U^n+V^m+1)^{p-1} U^{i-1} V^{j-1} \right)\neq 0$ equals the rank of the Cartier operator on $\Fnm^{(p)}$.
\end{proof}
This inspires the following definition. 
\begin{defn}
    Define $r_a(n,m,u)$ to be the number of pairs of positive integers $(i,j)$ with $mi+nj < nm$ such that there exists a solution $(h,k)$ with $0\leq h \leq u-1$ and $0\leq k \leq h$ to the system of congruences
    \begin{align*}
        \begin{cases}
            n(u-1-h)+i &\equiv 0 \mod u\\
            mk+j & \equiv 0 \mod u.
        \end{cases}
    \end{align*}
\end{defn}
For the function field $\Fnm^{(p)}$, we then have $a(\mathcal{F}_{n,m}) = g(\Fnm)-r_a(n,m,p)$ by Proposition~\ref{prop:cong_F_nm}, and we can therefore use this system of congruences to obtain several results about the $a$-number. 
In the following, it will be useful to note that if $M=\lcm(n,m)$, then $mi+nj<nm \iff \tfrac{M}{n}i + \tfrac{M}{m}j < M$.

\begin{theorem}\label{thm:p_mod_m}
    Let $n,m$ be positive integers, $M=\lcm(n,m)$ and $u$ a positive integer with $\gcd(M,u)=1$. Write $u=s+tM$ for integers $s,t$ such that $s\in\{1,\ldots,M-1\}$ and $t\geq 0$. 
    Then for a pair of positive integers $(i,j)$ such that $mi+nj < nm$, there exists a solution $(h,k)$ with $0\leq h \leq u-1$ and $0\leq k \leq h$ to 
    \begin{align}\label{congruence1}
        \begin{cases}
            n(u-1-h)+i &\equiv 0 \mod u\\
            mk+j & \equiv 0 \mod u 
        \end{cases},
    \end{align}
    if and only if there exists a solution $(\bar{h},\bar{k})$ with $0\leq \bar{h} \leq s-1$ and $0\leq \bar{k} \leq \bar{h}$ to 
    \begin{align}\label{congruence2}
        \begin{cases}
            n(s-1-\bar{h})+i &\equiv 0 \mod s\\
            m\bar{k}+j & \equiv 0 \mod s 
        \end{cases}    
    \end{align}
\end{theorem}
\begin{proof}
    Assume first that there exists a solution $(h,k)$ to the system~\eqref{congruence1} for $(i,j)$, and let $n_1,n_2$ be the integers such that
    \begin{align}\label{eq:ns_1}
         n(u-1-h)+i=n_1u \text{ and } mk+j=n_2u.
    \end{align}
    We will then find a solution $(\bar{h}, \bar{k})$ to the system~\eqref{congruence2} for $(i,j)$. We see that
    \begin{align*}
        n(u-1-h)+i=n_1u 
        \iff n(s-1-(h+n_1t\tfrac{M}{n}-tM))+i=n_1s,
    \end{align*}
    and 
    \begin{align*}
        mk+j=n_2u 
        \iff m(k-n_2t\tfrac{M}{m})+j=n_2s.
    \end{align*}
    Thus if we can show that $\bar{h}=h+n_1t\tfrac{M}{n}-tM$ and $\bar{k}=k-n_2t\tfrac{M}{m}$ satisfy $0\leq \bar{k}\leq\bar{h}\leq s-1$, then we have the solution $(\bar{h},\bar{k})$ to~\eqref{congruence2} that we claimed. 
    We first get that $m\bar{k} = n_2s-j$, so $n_2s-j$ is a multiple of $m$ and $\bar{k}\geq 0$ becomes equivalent to showing that $n_2s-j> -m$. Since $j<m$, we indeed get $n_2s-j\geq-j> -m$, and we have shown the first bound. 
    Then note that $\bar{h}-\bar{k} = h+n_1t\tfrac{M}{n}-tM - k + n_2t\tfrac{M}{m}$, so the next thing we need to show is
    \begin{align*}
        h - k \geq tM - t(n_1 \tfrac{M}{n}+ n_2\tfrac{M}{m}).
    \end{align*}
    Combining the two equations from~\eqref{eq:ns_1}, gives us 
    \begin{align}\label{eq:M(u-1-(h-k))}
        M(u-1-(h-k)) =  (n_1 \tfrac{M}{n} + n_2 \tfrac{M}{m})u - (i\tfrac{M}{n} + j \tfrac{M}{m}),
    \end{align}
    and by substituting $u=s+tM$, we get that
    \begin{align}\label{eq:h-k}
        h-k &=  u-1 -\frac{(n_1 \tfrac{M}{n} + n_2 \tfrac{M}{m})u - (i\tfrac{M}{n} + j \tfrac{M}{m})}{M}\\ \notag 
        &=  tM - t(n_1 \tfrac{M}{n} + n_2 \tfrac{M}{m}) + s-1 -\frac{(n_1 \tfrac{M}{n} + n_2 \tfrac{M}{m})s - (i\tfrac{M}{n} + j \tfrac{M}{m})}{M}.
    \end{align}
    Thus if we can show that $s+\frac{(i\tfrac{M}{n} + j \tfrac{M}{m}) - (n_1 \tfrac{M}{n} + n_2 \tfrac{M}{m})s}{M} \geq 1$, we have obtained the bound. 
    From Equation~\eqref{eq:M(u-1-(h-k))} and $i\tfrac{M}{n} + j \tfrac{M}{m} < M $, we get 
    \begin{align*}
        (n_1 \tfrac{M}{n} + n_2 \tfrac{M}{m})u -M < M(u-1-(h-k)) \leq M(u-1),
    \end{align*}
    in particular $(n_1 \tfrac{M}{n} + n_2 \tfrac{M}{m}) < M$, and we get that
    \begin{align*}
        s+\frac{(i\tfrac{M}{n} + j \tfrac{M}{m}) - (n_1 \tfrac{M}{n} + n_2 \tfrac{M}{m})s}{M}=
        \frac{(i\tfrac{M}{n} + j \tfrac{M}{m}) - (n_1 \tfrac{M}{n} + n_2 \tfrac{M}{m}-M)s}{M}
        > 0.
    \end{align*}
    Since $s+\frac{(i\tfrac{M}{n} + j \tfrac{M}{m}) - (n_1 \tfrac{M}{n} + n_2 \tfrac{M}{m})s}{M}$ is an integer by Equation~\eqref{eq:h-k}, this shows the bound. 
    We may obtain the last bound $\bar{h}\leq s-1$ by a series of manipulations, namely 
    \begin{align*}
        \bar{h} 
        &= h+n_1t\tfrac{M}{n}-tM 
        = h + \frac{(n(u-1-h)+i)t\tfrac{M}{n}-tMu}{u}\\
        &= h + \frac{(-1-h)(u-s)+it\frac{M}{n}}{u}
        = -1 + s\frac{(h+1)}{u} + \frac{it\frac{M}{n}}{u}\\
        &< -1 + s + 1 
        \leq s-1.
    \end{align*}
    Note that in the above, we used that $\frac{(h+1)}{u}\leq 1$ and $\frac{it\frac{M}{n}}{u} < \frac{tM}{u}\leq 1$. 
    We then need to show that a solution $(\bar{h},\bar{k})$ to the system~\eqref{congruence2} for $(i,j)$ gives rise to a solution $(h,k)$ to the system~\eqref{congruence1}. Now, let $n_1,n_2$ denote the integers such that $n(s-1-\bar{h})+i=n_1s$ and $m\bar{k}+j=n_2s$. Then 
    \begin{align*}
        n(s-1-\bar{h})+i=n_1s
        \iff n(u-1-(\bar{h}+tM-n_1t\tfrac{M}{n}))+i=n_1u
    \end{align*}
    and 
    \begin{align*}
        m\bar{k}+j=n_2s
        \iff m(\bar{k}+n_2t\tfrac{M}{m})+j=n_2u.
    \end{align*}
    We will then show the bound $h=\bar{h}+tM-n_1t\tfrac{M}{n}$ and $k=\bar{k}+n_2t\tfrac{M}{m}$ satisfies $k\leq h \leq p-1$. As before, we will need that $n_2\tfrac{M}{m}< M-n_1\tfrac{M}{n}$, so consider 
    \begin{align*}
        n_2\tfrac{M}{m} = \frac{\tfrac{M}{m}(m\bar{k}+j)}{s} = \frac{M\bar{k}+\tfrac{M}{m}j}{s}
    \end{align*}
    and
    \begin{align*}
        M-n_1\tfrac{M}{n} 
        = M - \left(\frac{n(s-1-\bar{h})+i}{s}\right)\frac{M}{n}
        = \frac{M\bar{h} + M - i\tfrac{M}{n}}{s}.
    \end{align*}
    Since $\bar{k}\leq \bar{h}$ and $\tfrac{M}{m}j < M - i\tfrac{M}{n}$, we indeed have that $n_2\tfrac{M}{m} < M-n_1\tfrac{M}{n}$, and it follows immediately that 
    \begin{align*}
        k=\bar{k}+n_2t\tfrac{M}{m}
        < \bar{h} + tM-n_1t\tfrac{M}{n}
        = h.
    \end{align*}
    Finally, we get that
    \begin{align*}
        h &= \bar{h} + t(M-n_1\tfrac{M}{n}) 
        = \bar{h} + \frac{tM(\bar{h}+1) - i\tfrac{M}{n}t}{s} \\
        &= \bar{h} + \frac{(u-s)(\bar{h}+1) - i\tfrac{M}{n}t}{s} 
        =  -1 + u\frac{(\bar{h}+1)}{s} - \frac{i\tfrac{M}{n}t}{s} \\
        & \leq u-1,
    \end{align*}
    which concludes the proof.
\end{proof}
Another interesting result on these systems of congruences is the following. 
\begin{theorem}\label{thm:p_inv}
    Let $n,m$ be positive integers, $M=\lcm(n,m)$, and $u$ a positive integer with $\gcd(M,u)=1$. Let $u^{-1}$ denote $(u^{-1}\mod M)$. Then for a pair of positive integers $(i,j)$ such that $mi+nj<nm$, there exists a solution $(h,k)$ with $0\leq h \leq u-1$ and $0\leq k \leq h$ to 
    \begin{align}\label{congruence1b}
        \begin{cases}
            n(u-1-h)+i &\equiv 0 \mod u\\
            mk+j & \equiv 0 \mod u.
        \end{cases}
    \end{align}
    if and only if there exists a solution $(\bar{h},\bar{k})$ with $0\leq \bar{h} \leq u^{-1}-1$ and $0\leq \bar{k} \leq \bar{h}$ to 
    \begin{align}\label{congruence2b}
        \begin{cases}
            n(u^{-1}-1-\bar{h})+\bar{i} &\equiv 0 \mod u^{-1}\\
            m\bar{k}+\bar{j} & \equiv 0 \mod u^{-1}
        \end{cases}    
    \end{align}
    where $(\bar{i},\bar{j})$ is a pair of uniquely determined positive integers such that $m\bar{i}+n\bar{j} < nm$. 
\end{theorem}
\begin{proof}
    Assume that $(h,k)$ is a solution to the system~\eqref{congruence1b} for $(i,j)$. Then we will show that there exists a pair $(\bar{h},\bar{k)}$ that is a solution to the system~\eqref{congruence2b} for a unique pair $(\bar{i},\bar{j})$ and that indeed $\bar{i},\bar{j}$ satisfies $m\bar{i} +n\bar{j}<mn$. By the symmetry of $u$ and $u^{-1}$, the other implication will then follow. 
    Define $n_1,n_2$ as the integers such that
    \begin{align}\label{eq:ns_2}
        n(u-1-h)+i=n_1u \text{ and }
        mk+j = n_2u,
    \end{align}
    and let $t$ be the integer satisfying $uu^{-1}=1+tM$. Then we have
    \begin{align*}
        n(u^{-1} - 1 - (tM-hu^{-1}-n_1 t \tfrac{M}{n})) + n_1 = iu^{-1}
    \end{align*}
    and
    \begin{align*}
        m(-ku^{-1}+n_2t\tfrac{M}{m})+n_2 = ju^{-1}.
    \end{align*}
    We see that $\bar{h}=tM-hu^{-1}-n_1 t \tfrac{M}{n}$, $\bar{k}=-ku^{-1}+n_2t\tfrac{M}{m}$, and $(\bar{i},\bar{j})=(n1,n2)$. 
    Note that
    \begin{align*}
        mn_1 +nn_2 
        = \frac{m(n(u-1-h)+i) + n(mk+j)}{u} 
        = mn-\frac{mn(h+1-k)-(mi+nj)}{u} 
        <mn,
    \end{align*}
    since $mn(h+1-k)\geq mn>mi+nj$, so it indeed makes sense to consider a solution to~\eqref{congruence2b} for $(\bar{i},\bar{j})$.
    Now, $m\bar{k}=ju^{-1}-n_2$, so $ju^{-1}-n_2$ is a multiple of $m$, and we see that if we show $ju^{-1}-n_2>-m$, then $\bar{k}\geq 0$. We have just shown that $mn_1 +nn_2<mn$, in particular $n_2<m$, so indeed $ju^{-1}-n_2\geq -n_2>-m$, and we get that $\bar{k}\geq 0$.     
    To show that $\bar{k}\leq \bar{h}$, we write
    \begin{align*}
        \bar{h} -\bar{k} = tM - hu^{-1} - n_1 t\tfrac{M}{n} + ku^{-1} - n_2 t \tfrac{M}{m} 
    \end{align*}
    and we see that it is in fact equivalent to showing that
    \begin{align*}
        (h-k)u^{-1} \leq t(M - n_1 \tfrac{M}{n} - n_2 \tfrac{M}{m}).
    \end{align*}
    Combining the equations from~\eqref{eq:ns_2}, we get that
    \begin{align*}
        u(n_1 \tfrac{M}{n} + n_2 \tfrac{M}{m}) 
        = M(u-1-(h-k)) + \tfrac{M}{n}i + \tfrac{M}{m}j,
    \end{align*}
    so we can write
    \begin{align*}
        h-k = u-1  + \frac{\tfrac{M}{n}i + \tfrac{M}{m}j - u\left(\tfrac{M}{n}n_1+\tfrac{M}{m}n_2\right)}{M}. 
    \end{align*}
    In particular, we get
    \begin{align*}
        (h-k)u^{-1} 
        &= uu^{-1}- u^{-1}  + \frac{\left(\tfrac{M}{n}i + \tfrac{M}{m}j\right)u^{-1} - uu^{-1}\left(\tfrac{M}{n}n_1+\tfrac{M}{m}n_2\right)}{M}\\
        &= t(M - n_1 \tfrac{M}{n} - n_2 \tfrac{M}{m}) +1  - \frac{\left(M-\left(\tfrac{M}{n}i + \tfrac{M}{m}j\right)\right)u^{-1} + \left(\tfrac{M}{n}n_1+\tfrac{M}{m}n_2\right)}{M}\\
        &\leq t(M - n_1 \tfrac{M}{n} - n_2 \tfrac{M}{m}),
    \end{align*}
    which shows that $\bar{k}\leq \bar{h}$. The last bound follows from a series of manipulations;    
    \begin{align*}
        \bar{h} &= tM-hu^{-1}-n_1t\tfrac{M}{n}
        = \frac{tMu - (n(u-1-h)+i) t \tfrac{M}{n}}{u}-hu^{-1}\\
        &= u^{-1} - \frac{h+1 + it \tfrac{M}{n}}{u} 
        \leq u^{-1} - 1.
    \end{align*}
    Finally, we will show that
    \begin{align*}
        (i,j) \mapsto (n_1,n2) = \left( \frac{n(u-1-h)+i}{u}, \ \frac{mk+j}{u} \right)
    \end{align*}
    is injective. Assume therefore that $\frac{n(u-1-h_1)+i_1}{u} = \frac{n(u-1-h_2)+i_2}{u}$ for integers $i_1,i_2$. Then $n(u-1-h_1)+i_1 = n(u-1-h_2)+i_2$, and $i_1 \equiv i_2 \mod n$. In particular $i_1=i_2$, since $0<i_1,i_2 < n$. The argument is similar for the second entry. 
    This concludes the proof.    
\end{proof}

Theorems \ref{thm:p_mod_m} and~\ref{thm:p_inv} imply the following relation between $a$-numbers for function fields over base fields of different characteristics. 

\begin{corollary}\label{cor:ppp}
    Let $n,m$ be positive integers and $M=\lcm(n,m)$. For primes $p, \bar{p},\tilde{p}$ such that $\bar{p} \equiv p \mod M$ and $\tilde{p} \equiv p^{-1} \mod M$, we have 
    \begin{align*}
        a(\Fnm^{(p)}) = a(\Fnm^{(\bar{p})}) = a(\Fnm^{(\tilde{p})}).
    \end{align*}
\end{corollary}

\begin{theorem}\label{thm:m1+m2}
    Let $e_1$ be a positive integer, $p_1$ a prime, $q_1=p_1^{e_1}$, and $m,n$ divisors of $q_1+1$. Denote by $M=\lcm(n,m)$ and $\mathcal{F}_{n,m}^{(p_1)}$ the maximal Fermat function field given by $U^n+V^m+1=0$. Let $m_1=p_1 \mod M$ and $m_2=M-m_1$. Then $r_a(n,m,m_1)+r_a(n,m,m_2)=g(\Fnm^{(p_1)})$.
\end{theorem}
\begin{proof}
    From Corollary~\ref{cor:fermat_type_C(w_ij)}, we have that for a pair of positive integers $(i,j)$ with $mi+nj<mn$ the Cartier operator $\mathcal{C}$ maps a holomorphic differential $\omega_{i,j}$ on $\Fnm^{(p_1)}$ to another holomorphic differential $\omega_{\ti_1,\jt_1}$ on $\Fnm^{(p_1)}$ if and only if $\ti_1=n\cl{\frac{ip_1^{e-1}}{n}} -ip_1^{e-1}$ and $\jt_1 = m\cl{\frac{jp_1^{e-1}}{m}} -jp_1^{e-1}$ satisfies the bound that $m\ti_1+n\jt_1 < nm$. This implies that the number of pairs $(i,j)$ such that $m\ti_1+n\jt_1 < nm$ is equal to the rank of $\mathcal{C}$ on $\Fnm^{(p_1)}$, i.e. $r_a(n,m,p_1)$. 
    The idea is now that we fix a pair $(i,j)$ and then consider the corresponding $\omega_{\ti_1,\jt_1}$ on $\Fnm^{(p_1)}$ and $\omega_{\ti_2,\jt_2}$ on $\Fnm^{(p_2)}$ with $p_2 \equiv m_2 \mod M$ and $n,m$ dividing $p_2^{e_2}+1$ for some exponent $e_2\in\mathbb{N}$. 
    We will show that $m\ti_1+n\jt_1 < nm \iff  m\ti_2+n\jt_2 \geq nm$, and this will imply that $r_a(n,m,m_1) + r_a(n,m,m_2)=g$.

    Since $m,n$ divides of $p_1^{e_1}+1$, we have that $p_1^{e_1} \equiv -1 \mod m$ and $p_1^{e_1} \equiv -1 \mod n$. Define $u,v$ as the integers such that $\gcd(n,m)=um+vn$. Then $p_1\ p_1^{e_1-1} \equiv p_1^{e_1} \equiv \frac{(-1)vn + (-1)um}{\gcd(n,m)}\equiv -1 \mod M$, and  
    \begin{align*}
        p_1^{e_1-1} \equiv -p_1^{-1} \equiv -m_1^{-1} \equiv m_2^{-1} \mod M.
    \end{align*}
    The same congruences holds modulo $m$ and $n$. For simplicity, let $m_1^{-1},m_2^{-1}$ denote  $m_1^{-1}\mod M$ and $m_2^{-1} \mod M$ respectively. Then
    \begin{align*}
        \ti_1
        = n\cl{\frac{ip_1^{e_1-1}}{n}} -ip_1^{e_1-1}
        = n\cl{\frac{i(p_1^{e_1-1}-m_2^{-1})+im_2^{-1}}{n}} -ip_1^{e_1-1}
        = n\cl{\frac{im_2^{-1}}{n}}-im_2^{-1}.
    \end{align*}
    Similarly one can obtain $\ti_2= n\cl{\frac{im_1^{-1}}{n}}-im_1^{-1}$.
    Since $m_2^{-1} = M-m_1^{-1}$ and $\cl{\frac{-im_1^{-1}}{n}}=-\fl{\frac{im_1^{-1}}{n}}=-\cl{\frac{im_1^{-1}}{n}}+1$, we then get
    \begin{align*}
        \ti_1 &= n \cl{\frac{i(M-m_1^{-1})}{n}} - i(M-m_1^{-1})
        = iM + n \cl{\frac{-im_1^{-1}}{n}} - iM+im_1^{-1}\\
        &= n - \left(n\cl{\frac{im_1^{-1}}{n}}-im_1^{-1}\right)
        = n-\ti_2.
    \end{align*}
    A completely analogue argument gives us that $\jt_1 = m - \left(m\cl{\frac{jm_1^{-1}}{m}}-jm_1^{-1}\right) = m-\jt_2$. 
    Then $m\ti_1 + n\jt_1 = m(n-\ti_2) + n(m-\jt_2) = 2nm - (m\ti_2+n\jt_2)$, and 
    and we see that 
    \begin{align*}
        m\ti_1 + n\jt_1 < nm \iff m\ti_2+n\jt_2 > nm.
    \end{align*}
    Finally, we need to show that $m\ti_1 + n\jt_1 = nm$ cannot happen. Assume contrarily that
    \begin{align*}
        m\ti_1 + n\jt_1 
        = mn\left(\cl{\frac{im_2^{-1}}{n}}+\cl{\frac{jm_2^{-1}}{m}}\right) -m_2^{-1} (mi+nj) =mn.
    \end{align*}
    However, this would imply that $mn\mid (m_2^{-1} (mi+nj))$, a contradiction since $\gcd(nm,m_2^{-1})=1$ and $mi+nj<nm$. 
\end{proof}

\begin{corollary}\label{cor:m1+m2}
    Let the situation be as in Theorem~\ref{thm:m1+m2}. Then if $p_2$ is a prime such that $\Fnm^{(p_2)}$ is maximal and $p_2\equiv m_2 \mod M$, we have 
    \begin{align*}
        a(\Fnm^{(p_1)}) + a(\Fnm^{(p_2)}) = g(\Fnm^{(p_1)}).
    \end{align*}
\end{corollary}

We finish this section with proving a relation between the $a$-number and the genus of the maximal Fermat type function field whenever $q=p^2$. This agrees with what was already known about the Hermitian function field (see~\cite[Proposition 14.10]{Gross} and~\cite[Theorem 3.2]{MoS}). 

\begin{corollary}\label{cor:fermat-type_a=g/2}
    Let $p$ be a prime, $q=p^2$ and $n,m$ divisors of $q+1$ such that $\mathcal{F}_{n,m}^{(p)}$ is maximal.
    Then $a(\mathcal{F}_{n,m}^{(p)})=g(\mathcal{F}_{n,m}^{(p)})/2$.
\end{corollary}
\begin{proof}
    Let $M=\lcm(n,m)$ and let $p^{-1}$ denote $(p^{-1} \mod M)$. Since $n,m$ divides $(q+1)$, we have that $p^2+1\equiv 0 \mod M$, so $p^{-1}\equiv-p \mod M$. In particular $p^{-1}+p\equiv 0 \mod M$, and Theorem~\ref{thm:m1+m2} 
    thus implies that $r_a(n,m,p)+r_a(n,m,p^{-1})=g(\mathcal{F}_{n,m}^{(p)})$.
    Theorem~\ref{thm:p_inv} then gives us that $r_a(n,m,p)=g(\mathcal{F}_{n,m}^{(p)})/2$, and the result follows.
\end{proof}

\section{Various results about the $a$-number of the Hurwitz function field.}
In this section, we will show that the relations between $a$-numbers of Fermat type function fields of different characteristic found in the previous section in fact also hold for Hurwitz function fields of different characteristic. We will therefore also here use the superscript notation $\mathcal{H}_n^{(p)}$ to denote the Hurwitz function field defined by $S^nT+T^n+S=0$ over a field of characteristic $p>0$. 
Recall that $N=n^2-n+1$. If we assume $\mathcal{H}_n$ is maximal, then $N|(q+1)$, in particular the Fermat function field $\mathcal{F}_N$ above is also maximal, and the first result already follows. 
\begin{corollary}
    Let $n$ be a positive integer such that $\mathcal{H}_n$ is maximal with characteristic $p_1>0$. Let $m_1=p_1 \mod N$ and $m_2=N-m_1$. 
    Then if $p_2$ is a prime such that $\mathcal{H}_n^{(p_2)}$ is maximal, and $p_2\equiv m_2 \mod N$, we have 
    \begin{align*}
        a(\mathcal{H}_n^{(p_1)}) + a(\mathcal{H}_n^{(p_2)}) = g(\mathcal{H}_n^{(p_1)}).
    \end{align*}
\end{corollary}
For the rest of the section, we will not assume that $\mathcal{H}_n$ is maximal. 
In the article~\cite{MoS} they gave an expression for the $a$-number of the Hurwitz function field which inspires the following definition. 
\begin{defn}
    Define $r_a(n,u)$ to be the number of pairs $(i,j)$ with $0\leq i+j\leq n-2$ such that there exists a solution $(h,k)$ with $0\leq k \leq h \leq u-1$ to the system of congruences
\begin{align*}
    \begin{cases}
        nk-h+i \equiv 0 & \mod u\\
        n(h-k) + k + j + 1 \equiv 0 & \mod u.
    \end{cases}
\end{align*}
\end{defn}
In particular, it was shown that $r_a(n,p)$ equals the rank of the Cartier operator on $\mathcal{H}_n^{(p)}$.

\begin{theorem}
    Let $u$ be a positive integer with $\gcd(u,N)=1$ and $s,t$ nonnegative integers such that $u=s+tN$ and $1\leq s \leq N-1$. 
    Then for each pair $(i,j)$ with $0\leq i+j\leq n-2$ there exists a solution $(h,k)$ with $0\leq k \leq h \leq u-1$ to 
    \begin{align}\label{Hurwitz:cong1a}
        \begin{cases}
            nk-h+i \equiv 0 & \mod u\\
            n(h-k) + k + j + 1 \equiv 0 & \mod u
        \end{cases}
    \end{align}
    if and only if there exists a solution $(\bar{h}, \bar{k})$ with $0 \leq \bar{k} \leq \bar{h} \leq s-1$ to 
    \begin{align}\label{Hurwitz:cong1b}
        \begin{cases}
            n\bar{k}-\bar{h}+i \equiv 0 & \mod s\\
            n(\bar{h}-\bar{k}) + \bar{k} + j + 1 \equiv 0 & \mod s.
        \end{cases}
    \end{align}
\end{theorem}
\begin{proof}
    We will first assume that there is a solution to the system~\eqref{Hurwitz:cong1a} for some $(i,j)$, so define $n_1,n_2$ such that $nk-h+i=n_1u$ and $n(h-k) + k + j + 1 =n_2u$. 
    Then $nk-h+i-n_1tN=n_1s$ and $n(h-k) + k + j + 1 - n_2tN=n_2s$, and we consider 
    \begin{align*}
        \begin{cases}
            n\bar{k}-\bar{h} &= nk-h-n_1tN,\\
            n(\bar{h}-\bar{k}) + \bar{k} &= n(h-k) + k - n_2tN,
        \end{cases}
    \end{align*}
    which has the solution
    \begin{align*}
        \bar{h} = h+n_1t - nt(n_1+n_2) \text{ and }
        \bar{k} = k-nn_1t-n_2t.
    \end{align*}
    We need to show that $0\leq \bar{k} \leq \bar{h} \leq s-1$, so note first that $n\bar{k}=n_1s+\bar{h}-i$, so $n$ divides $n_1s+\bar{h}-i$. Thus if $n_1s+\bar{h}-i>-n$, then $\bar{k}\geq 0$ follows, and indeed we have $n_1s+\bar{h}-i\geq -i>-n$. For the bound $\bar{k}\leq \bar{h}$, we get that    
    \begin{align*}
        \bar{h}-\bar{k} = h-k + (n_1+n_2)t-nn_2t,
    \end{align*}
    so we will show that $nn_2t-(n_1+n_2)t \leq h-k$, and then $\bar{h}-\bar{k}\geq 0$ follows.
    A series of computations show that 
    \begin{align}\label{eq:(n1+n2)u}
        (n_1+n_2)u = nh-h+k+i+j+1,
    \end{align}
    and 
    \begin{align}\label{eq:nn_2u-(n_1+n_2)u}
        nn_2u-(n_1+n_2)u 
        &= N(h-k)+nj + (n-i-j-1). 
    \end{align}
    Thus, 
    \begin{align*}
        nn_2t-(n_1+n_2)t
        &=\frac{Nt(h-k)+ntj + t(n-i-j-1)}{u}\\
        &= h-k + \frac{-s(h-k)+ntj + t(n-i-j-1)}{u}\\
        &< h-k + \frac{tN}{u}
    \end{align*}
    However, since $\frac{tN}{u}<1$, and all the values are integer, we may conclude that $nn_2t-(n_1+n_2)t \leq h-k$, and $\bar{h}-\bar{k}\geq 0$. Now, in order to show that $\bar{h}\leq s-1$, we first need
    \begin{align}\label{eq:n_1u-n(n1+n2)u}
        n_1u-n(n_1+n_2)u = -Nh+i-n(i+j+1).
    \end{align}
    Then we immediately get
    \begin{align*}
        \bar{h} 
        &= \frac{hu-Nth+it-nt(i+j+1)}{u}
        = \frac{hs+it-nt(i+j+1)}{u}\\
        &< \frac{us+it-nt(i+j+1)}{u}
        \leq s -\frac{nt(j+1)}{u}
        \leq s. 
    \end{align*}
    Now, we assume that there exists a solution $(\bar{h},\bar{k})$ to the system~\eqref{Hurwitz:cong1b} for some $(i,j)$. Let therefore $n_1,n_2$ be defined such that $n\bar{k}-\bar{h}+i=n_1s$ and $n(\bar{h}-\bar{k}) + \bar{k} + j + 1 =n_2s$. We see 
    that $n\bar{k}-\bar{h}+i+n_1tN=n_1u$ and $n(\bar{h}-\bar{k}) + \bar{k} + j + 1 + n_2tN =n_2s$, and then we have the solution 
    \begin{align*}
        h = \bar{h} + nt(n_1+n_2)-n_1t \text{ and } k = \bar{k}+nn_1t+n_2t.
    \end{align*}
    Clearly $k\geq 0$, so we just need to show $k \leq h \leq u-1$. We see that
    \begin{align*}
        h-k = \bar{h}-\bar{k} + ntn_2 -(n_1+n_2)t,
    \end{align*}
    so $k \leq h$ will follow from showing $(n_1+n_2)t - ntn_2 \leq \bar{h}-\bar{k}$. We can use the expression from Equation~\eqref{eq:nn_2u-(n_1+n_2)u} with the new solution, and we get
    \begin{align*}
        (n_1+n_2)t - ntn_2 
        &= \frac{-Nt(\bar{h}-\bar{k})-ntj-t(n-i-j-1)}{s}\\
        &= \bar{h}-\bar{k}-\frac{u(\bar{h}-\bar{k})+ntj+t(n-i-j-1)}{s}\\
        &\leq \bar{h}-\bar{k}.
    \end{align*}
    Finally, we use the expression from Equation~\eqref{eq:n_1u-n(n1+n2)u} with the new solution and get
    \begin{align*}
        h 
        &= \frac{\bar{h}s+Nt\bar{h}-it+nt(i+j+1)}{s}
        = \frac{\bar{h}u-it+nt(i+j+1)}{s}\\
        &\leq u + \frac{-u-it+nt(n-1)}{s}
        < u - \frac{u-Nt}{s} = u-1. 
    \end{align*}
    This concludes the proof.
\end{proof}
Like for the Fermat type function field, we have a similar result relating the solutions to the system of congruences modulo $u$ and modulo its inverse. 
\begin{theorem}
    Let $u^{-1}$ denote $(u^{-1}\mod N)$ and define $t$ such that $uu^{-1}=1+tN$. 
    Then for each pair of integers $(i,j)$ with $0\leq i+j\leq n-2$ there exists a solution $(h,k)$ with $0\leq k \leq h \leq u-1$ to 
    \begin{align}\label{Hurwitz:cong2a}
        \begin{cases}
            nk-h+i \equiv 0 & \mod u\\
            n(h-k) + k + j + 1 \equiv 0 & \mod u
        \end{cases}
    \end{align}
    if and only if there exists a uniquely determined pair of integers $(\bar{i},\bar{j)}$ with $0\leq \bar{i}+\bar{j}\leq n-2$ and a solution $(\bar{h}, \bar{k})$ with $0 \leq \bar{k} \leq \bar{h} \leq u^{-1}-1$ to 
    \begin{align}\label{Hurwitz:cong2b}
        \begin{cases}
            n\bar{k}-\bar{h}+\bar{i} \equiv 0 & \mod u^{-1}\\
            n(\bar{h}-\bar{k}) + \bar{k} + \bar{j} + 1 \equiv 0 & \mod u^{-1}.
        \end{cases}
    \end{align}
\end{theorem}
\begin{proof}
    Assume that $(h,k)$ is a solution to the system~\eqref{Hurwitz:cong2a} for $(i,j)$ and define $n_1,n_2$ such that $nk-h+i=n_1u$ and $n(h-k) + k + j + 1 =n_2u$. Then by multiplying by $u^{-1}$, we see that
    \begin{align*}
        -nku^{-1}+hu^{-1}+  n_1 + tNn_1 =iu^{-1},
    \end{align*}
    and 
    \begin{align*}
       -n(hu^{-1}-ku^{-1}) - ku^{-1} + tNn_2 + (n_2 -1) + 1 = (j+1)u^{-1}.
    \end{align*}
    Let $(\bar{i},\bar{j})=(n_1, n_2-1)$ and, with the obvious notion of $\bar{n}_1, \bar{n}_2$, let $(\bar{n}_1, \bar{n}_2)=(i,j+1)$. Before finding a solution, we will show that this choice satisfies $0\leq \bar{i} + \bar{j} \leq n-2$ or specifically that $n_1+n_2\leq n-1$. We can use the expression obtained in Equation~\eqref{eq:(n1+n2)u}, and we see
    \begin{align*}
        n_1+n_2 
        = \frac{nh-(h-k)+i+j+1}{u}
        \leq \frac{n(u-1)+i+j+1}{u}
        = n - \frac{n-1 -i-j}{u}.
    \end{align*}
    Hence $n_1+n_2\leq n-1$. Then we can consider the system of equations
    \begin{align*}
        \begin{cases}
            n\bar{k}-\bar{h} &= -nku^{-1}+hu^{-1} + tNn_1,\\
            n(\bar{h}-\bar{k}) + \bar{k} &= -n(hu^{-1}-ku^{-1}) - ku^{-1} + tNn_2
        \end{cases}
    \end{align*}
    which has solution 
    \begin{align*}
        \bar{h} = nt(n_1+n_2)-hu^{-1}-n_1t \text{ and }
        \bar{k} = nn_1t - ku^{-1} + n_2t.
    \end{align*}
    We first note that $n\bar{k} = iu^{-1}+\bar{h} - n_1$, so $n$ divides $(iu^{-1}+\bar{h} - n_1)$, and since $n_1<n$ we get $iu^{-1}+\bar{h} - n_1 > -n$, in particular $iu^{-1}+\bar{h} - n_1\geq 0$ and thus $\bar{k}\geq 0$.
    Then in order to show that $\bar{k}\leq \bar{h}$ consider
    \begin{align*}
        \bar{h}-\bar{k} = nn_2t-(h-k)u^{-1}-(n_1+n_2)t.
    \end{align*}
    We see that the bound follows if we can show that 
    \begin{align*}
        h-k \leq \frac{nn_2t-(n_1+n_2)t}{u^{-1}} = \frac{nn_2u-(n_1+n_2)u}{1+tN}t.
    \end{align*}
    We may use the expression from~\eqref{eq:nn_2u-(n_1+n_2)u} and get that indeed
    \begin{align*}
        \frac{nn_2u-(n_1+n_2)u}{1+tN}t 
        &= \frac{N(h-k)+nj + (n-i-j-1)}{1+tN}t\\
        &> \frac{(Nt+1)(h-k)-(h-k)}{1+tN}\\
        &\geq h-k - \frac{u-1}{1+tN}.
    \end{align*}
    We conclude that $\frac{nn_2u-(n_1+n_2)u}{1+tN}t\geq h-k$. For the final bound, we use~\eqref{eq:n_1u-n(n1+n2)u} and see that
    \begin{align*}
        \bar{h} 
        & =  \frac{-hu^{-1}u + Nth - it +nt(i+j+1)}{u}
        = \frac{-h-it+nt(i+j+1)}{u}\\
        &\leq \frac{tn(n-1)}{u}
        < \frac{Nt}{u}
        < u^{-1}-\frac{1}{u}.
    \end{align*}
    The other implication follows by symmetry, and this concludes the proof.
\end{proof}

Finally, we get that also for the Hurwitz function field, we have the following interesting relation between function fields of different characteristic. 
\begin{corollary}
    Let $n$ be a positive integer and $N=n^2-n+1$. Let $p, \bar{p},\tilde{p}$ be primes such that $\bar{p} \equiv p \mod N$ and $\tilde{p} \equiv p^{-1} \mod N$ and none of $p, \bar{p},\tilde{p}$ divides $N$. Then 
    \begin{align*}
        a(\mathcal{H}_n^{(p)}) = a(\mathcal{H}_n^{(\bar{p})}) = a(\mathcal{H}_n^{(\tilde{p})}).
    \end{align*}
\end{corollary}

As a consequence, the result obtained for maximal Fermat type function fields in Corollary~\ref{cor:fermat-type_a=g/2} also holds for maximal Hurwitz function fields. 
\begin{corollary}
    Let $p$ be a prime, $q=p^2$ and $n$ a nonnegative integer such that $n^2-n+1$ divides $q+1$ so that $\mathcal{H}_n$ is maximal. Then $a(\mathcal{H}_n)=g(\mathcal{H}_n)/2$.
\end{corollary}

\section*{Acknowledgements}
The author is deeply grateful to Peter Beelen and Maria Montanucci for their useful input and helpful discussions. 
This work has been supported by Villum Fonden under Grant VIL52303.

\section*{AI tool disclosure}
AI assistance was used for literature search. The ideas, proofs and text is human generated.


\end{document}